\documentclass[leqno,11pt]{article}

\RequirePackage[amsthm,amsmath]{imsart}

\usepackage[mathscr]{eucal}
\usepackage{graphicx}
\usepackage{latexsym}
\usepackage{amssymb}
\usepackage{psfrag}
\usepackage{epsfig}
\usepackage{MnSymbol}
\usepackage{enumerate}
\usepackage{multirow}
\DeclareMathAlphabet{\mathpzc}{OT1}{pzc}{m}{it}

\include{psbox}

\usepackage{amsfonts}
\usepackage{graphicx}
\usepackage{amsfonts}
\usepackage{color}
\usepackage[usenames,dvipsnames]{xcolor}
\usepackage[textsize=small]{todonotes}
\usepackage{caption,subcaption}
\usepackage{soul}

\usepackage[numbers]{natbib}

\begin{document}

\renewcommand*{\bibfont}{\small}
	\newtheorem{proposition}{Proposition}[section]
	\newtheorem{theorem}[proposition]{Theorem}
	\newtheorem{corollary}[proposition]{Corollary}
	\newtheorem{lemma}[proposition]{Lemma}
	\newtheorem{conjecture}[proposition]{Conjecture}
	\newtheorem{question}[proposition]{Question}
	\newtheorem{definition}[proposition]{Definition}
	\newtheorem{comment}[proposition]{Comment}
	\newtheorem{algorithm}[proposition]{Algorithm}
	\newtheorem{assumption}[proposition]{Assumption}
	\newtheorem{condition}[proposition]{Condition}
	\numberwithin{equation}{section}
	\numberwithin{proposition}{section}

\captionsetup[table]{format=plain,labelformat=simple,labelsep=period}

\newcommand{\skp}{\vspace{\baselineskip}}
\newcommand{\noi}{\noindent}
\newcommand{\osc}{\mbox{osc}}
\newcommand{\lfl}{\lfloor}
\newcommand{\rfl}{\rfloor}

\theoremstyle{remark}
\newtheorem{example}{\bf Example}[section]
\newtheorem{remark}{\bf Remark}[section]

\newcommand{\img}{\imath}
\newcommand{\iy}{\infty}
\newcommand{\eps}{\varepsilon}
\newcommand{\del}{\delta}
\newcommand{\Rk}{\mathbb{R}^k}
\newcommand{\RR}{\mathbb{R}}
\newcommand{\spa}{\vspace{.2in}}
\newcommand{\V}{\mathcal{V}}
\newcommand{\E}{\mathbb{E}}
\newcommand{\I}{\mathbb{I}}
\newcommand{\PP}{\mathbb{P}}
\newcommand{\sgn}{\mbox{sgn}}
\newcommand{\ti}{\tilde}

\newcommand{\QQ}{\mathbb{Q}}

\newcommand{\XX}{\mathbb{X}}
\newcommand{\XXz}{\mathbb{X}^0}

\newcommand{\lan}{\langle}
\newcommand{\ran}{\rangle}
\newcommand{\llan}{\llangle}
\newcommand{\rran}{\rrangle}
\newcommand{\lf}{\lfloor}
\newcommand{\rf}{\rfloor}
\def\wh{\widehat}
\newcommand{\defn}{\stackrel{def}{=}}
\newcommand{\txb}{\tau^{\eps,x}_{B^c}}
\newcommand{\tyb}{\tau^{\eps,y}_{B^c}}
\newcommand{\tilxb}{\tilde{\tau}^\eps_1}
\newcommand{\pxeps}{\mathbb{P}_x^{\eps}}
\newcommand{\non}{\nonumber}
\newcommand{\dist}{\mbox{dist}}

\newcommand{\Om}{\mathnormal{\Omega}}
\newcommand{\om}{\omega}
\newcommand{\vph}{\varphi}
\newcommand{\Del}{\mathnormal{\Delta}}
\newcommand{\Gam}{\mathnormal{\Gamma}}
\newcommand{\Sig}{\mathnormal{\Sigma}}

\newcommand{\tilyb}{\tilde{\tau}^\eps_2}
\newcommand{\beq}{\begin{eqnarray*}}
\newcommand{\eeq}{\end{eqnarray*}}
\newcommand{\beqn}{\begin{eqnarray}}
\newcommand{\eeqn}{\end{eqnarray}}
\newcommand{\ink}{\rule{.5\baselineskip}{.55\baselineskip}}

\newcommand{\bt}{\begin{theorem}}
\newcommand{\et}{\end{theorem}}
\newcommand{\deps}{\Del_{\eps}}
\newcommand{\dbl}{\mathbf{d}_{\tiny{\mbox{BL}}}}

\newcommand{\be}{\begin{equation}}
\newcommand{\ee}{\end{equation}}
\newcommand{\ac}{\mbox{AC}}
\newcommand{\hs}{\tiny{\mbox{HS}}}
\newcommand{\BB}{\mathbb{B}}
\newcommand{\VV}{\mathbb{V}}
\newcommand{\DD}{\mathbb{D}}
\newcommand{\KK}{\mathbb{K}}
\newcommand{\HH}{\mathbb{H}}
\newcommand{\TT}{\mathbb{T}}
\newcommand{\CC}{\mathbb{C}}
\newcommand{\ZZ}{\mathbb{Z}}
\newcommand{\SSS}{\mathbb{S}}
\newcommand{\EE}{\mathbb{E}}
\newcommand{\NN}{\mathbb{N}}
\newcommand{\MM}{\mathbb{M}}

\newcommand{\clg}{\mathcal{G}}
\newcommand{\clb}{\mathcal{B}}
\newcommand{\cls}{\mathcal{S}}
\newcommand{\clc}{\mathcal{C}}
\newcommand{\clj}{\mathcal{J}}
\newcommand{\clm}{\mathcal{M}}
\newcommand{\clx}{\mathcal{X}}
\newcommand{\cld}{\mathcal{D}}
\newcommand{\cle}{\mathcal{E}}
\newcommand{\clv}{\mathcal{V}}
\newcommand{\clu}{\mathcal{U}}
\newcommand{\clr}{\mathcal{R}}
\newcommand{\clt}{\mathcal{T}}
\newcommand{\cll}{\mathcal{L}}
\newcommand{\clz}{\mathcal{Z}}
\newcommand{\clq}{\mathcal{Q}}
\newcommand{\clo}{\mathcal{O}}

\newcommand{\cli}{\mathcal{I}}
\newcommand{\clp}{\mathcal{P}}
\newcommand{\cla}{\mathcal{A}}
\newcommand{\clf}{\mathcal{F}}
\newcommand{\clh}{\mathcal{H}}
\newcommand{\N}{\mathbb{N}}
\newcommand{\Q}{\mathbb{Q}}
\newcommand{\bfx}{{\boldsymbol{x}}}
\newcommand{\bfa}{{\boldsymbol{a}}}
\newcommand{\bfh}{{\boldsymbol{h}}}
\newcommand{\bfs}{{\boldsymbol{s}}}
\newcommand{\bfm}{{\boldsymbol{m}}}
\newcommand{\bff}{{\boldsymbol{f}}}
\newcommand{\bfb}{{\boldsymbol{b}}}
\newcommand{\bfw}{{\boldsymbol{w}}}
\newcommand{\bfz}{{\boldsymbol{z}}}
\newcommand{\bfu}{{\boldsymbol{u}}}
\newcommand{\bfell}{{\boldsymbol{\ell}}}
\newcommand{\bfn}{{\boldsymbol{n}}}
\newcommand{\bfd}{{\boldsymbol{d}}}
\newcommand{\bfbeta}{{\boldsymbol{\beta}}}
\newcommand{\bfzeta}{{\boldsymbol{\zeta}}}
\newcommand{\bfnu}{{\boldsymbol{\nu}}}
\newcommand{\bfvarphi}{{\boldsymbol{\varphi}}}

\newcommand{\curvz}{{\bf \mathpzc{z}}}
\newcommand{\curvx}{{\bf \mathpzc{x}}}
\newcommand{\curvi}{{\bf \mathpzc{i}}}
\newcommand{\curvs}{{\bf \mathpzc{s}}}
\newcommand{\blip}{\mathbb{B}_1}
\newcommand{\loc}{\text{loc}}

\newcommand{\BM}{\mbox{BM}}

\newcommand{\tac}{\mbox{\scriptsize{AC}}}
%\newcommand{\beginsec}{
% %\setcounter{equation}{0}
% \setcounter{lemma}{0} \setcounter{theorem}{0}
% \setcounter{corollary}{0} \setcounter{definition}{0}
% \setcounter{example}{0} \setcounter{proposition}{0}
% \setcounter{condition}{0} \setcounter{assumption}{0}
% %\setcounter{conjecture}{0} \setcounter{problem}{0}
% \setcounter{remark}{0} }

%\numberwithin{equation}{section} \numberwithin{lemma}{section}

%\renewcommand{\bibfont}{\footnotesize}

\begin{frontmatter}
\title{Steady-State Behavior of Some Load Balancing Mechanisms in Cloud Storage Systems	% \thanks{Research  supported in part by the National Science Foundation (DMS-1004418, DMS-1016441, DMS-1305120) and the Army Research
% 	 Office (W911NF-10-1-0158, W911NF- 14-1-0331)}
}

 \runtitle{Steady-State Behavior in Cloud Storage Systems}

\begin{aug}
\author{Eric Friedlander\\
}
\end{aug}

\today

\skp

\begin{abstract}
In large storage systems, files are often coded across several servers to improve
reliability and retrieval speed. We consider a system of $n$ servers storing files using a Maximum Distance Separable code (cf. \cite{li2016mean}). Specifically, each file is stored in equally sized pieces
across $L$ servers such that any $k$ pieces can reconstruct the original file. File requests are routed using the Batch Sampling routing scheme. I.e. when a request
for a file is received, a centralized dispatcher routes the job into the $k$-shortest queues among the
$L$ for which the corresponding server contains a piece of the file being requested. 
We study the long time behavior of this class of load balancing mechanisms. 
In particular, it is shown that the ODE system that describes the mean field limit of the occupancy measure process has a unique fixed point which is stable. This fixed point corresponds to a distribution on $\NN_0$ of queue lengths  with tails that decay super-exponentially. Upper and lower bounds on the decay rate are provided. Finally, we show that the unique invariant measure of the Markov occupancy measure process converges to the Dirac measure concentrated at the unique fixed point of the ODE system, establishing the interchangeability of the $t\to\iy$ and $n\to\iy$ limits.

\noi {\bf AMS 2000 subject classifications:} 60K25, 60K30, 60K35, 90B22, 68M20, 90C31.

\noi {\bf Keywords}: Mean field approximations, stochastic networks, propagation of chaos, cloud storage systems, supermarket model, MDS coding, Power-of-d, interchange of limits, stability.
\end{abstract}

\end{frontmatter}

\section{Introduction}
In large data centers, files are often ``coded'' to improve reliability and retrieval speed. In such a system each file is broken down into smaller pieces and distributed across multiple servers. In this work, we consider load balancing mechanisms in a network storing files using a Maximum Distance Separable (MDS) code. We model such a system by considering a system of $n$ servers each maintaining its own first-in-first out (FIFO) queue. $I(n)$ files are stored in the system such that each file is stored on $L$ servers and any subset of size $k$ contains enough information to reconstruct the original file. A stream of jobs requesting files uniformly at random arrive from the outside and are then routed by a centralized dispatcher. We consider a routing scheme known as Batch Sampling (BS) (cf. \cite{li2016mean}) in which incoming file requests are routed to the $k$ shortest queues among the $L$ corresponding to servers containing the requested file. In the special case where $k=1$, this load balancing scheme reduces to the well-studied supermarket model (see \cite{vvedenskaya1996queueing,graham2000chaoticity}).

Batch sampling schemes of the form considered in this work were first studied in \cite{li2016mean} where the authors calculate the steady state $(t\to\iy)$ queue length distribution in the large system limit.
Recently, in \cite{budhiraja2017diffusion}, limit theorems over finite time horizons for the occupancy measure process associated with such systems were established.
Specifically, \cite{budhiraja2017diffusion} proved convergence in probability, on the path space, of the occupancy measure process to a system of ODE.
Furthermore, \cite{budhiraja2017diffusion} studied fluctuations from the deterministic limit by establishing a functional central limit theorem with the limit described through an infinite dimensional linear SDE.
The results in \cite{budhiraja2017diffusion} describe the asymptotic system behavior over any finite time horizon, however they do not address the long time behavior (e.g. stability, properties of the invariant distribution, etc.) of the occupancy process.
In this work our goal is to study the steady state properties of such systems which is of primary relevance when system behavior over a long time horizon is of interest.
We show that the ODE system indentified in \cite{budhiraja2017diffusion} that determines the law of large number (LLN) behavior of the occupancy measure process has a unique fixed point $\bar{u}$ which is stable. Namely, starting from an arbitrary initial condition, the solution to the ODE converges to this fixed point as $t\to\iy$. We also show that the queue length distribution corresponding to the fixed point has tails which decay super-exponentially extending this well known property of the supermarket model (i.e. $k=1$) to a general $k<L$. We give explicit upper and lower bounds (cf. Theorem \ref{thm:decay}) on these tails which are sharp in the sense that they coincide when $k=1$. 
Finally we prove an important interchange of limit property. 
In \cite{li2016mean}, it has been shown that queue length process $Q^n$ for the $n$-server system is positive recurrent and, thus, has a unique invariant probability measure.
This then implies that the occupancy measure process has a unique invariant distribution.
In this work we show that this invariant measure converges to $\del_{\bar{u}}$ in probability, as $n\to\iy$.
Roughly speaking, this result says that the limits $n\to\iy$ and $t\to\iy$ can be interchanged and, in particular, the fixed point of the ODE is a good approximation for the steady state behavior of the occupancy process for large $n$.

The organization of the paper is as follows. In Section \ref{sec:main} we present our main results. The fixed point of the ODE system from \cite{budhiraja2017diffusion}, which was first identified in \cite{li2016mean}, is given in \eqref{eqn:fixedPoint}. Theorem \ref{thm:decay} gives explicit upper and lower bounds on the rate of decay for the tail of the queue length distribution determined by \eqref{eqn:fixedPoint}. In Theorem \ref{thm:stableFluid} we show that \eqref{eqn:fixedPoint} is, in fact, a stable fixed point of the ODE \eqref{eqn:ODE}. Finally, Theorem \ref{thm:interchange} presents the interchange of limits property discussed above. The remainder of the paper is devoted to proving the above results. In Section \ref{sec:equivalenceproof} we prove a lemma which will be needed in the proof of Theorems \ref{thm:stableFluid} and \ref{thm:interchange}. Proof of Theorems \ref{thm:decay}, \ref{thm:stableFluid}, and \ref{thm:interchange} are then given in Sections \ref{sec:decayProof}, \ref{sec:stability}, and \ref{sec:interchange}, respectively.

\subsection{Related Work}
Load balancing mechanisms similar to the type considered here have been studied in many works. Specifically, the Join-the-Shortest-Queue (JSQ), Join-the-Idle-Queue (JIQ), and Power-of-$d$ (also known as the supermarket model)  routing schemes have garnered quite a bit of attention (see \cite{vvedenskaya1996queueing, eschenfeldt2015join,mukherjee2016universality,mitzenmacher2001power,bramson2012asymptotic,stolyar2015pull, graham2000chaoticity,mukherjee2016asymptotic, mukherjee2017asymptotically} and  references therein). The papers \cite{vvedenskaya1996queueing, mitzenmacher2001power} were the first to study the tail behavior of the fixed point of the ODE system associated with the Power-of-$d$ routing scheme, showing that in steady-state the fraction of queues with lengths exceeding $m$
decay super-exponentially in $m$, a large improvement over the exponential rate for the setting where jobs are routed to servers uniformly at random.
Later works on a similar theme include \cite{mukherjee2016universality, mukherjee2016asymptotic, mukherjee2017asymptotically}. 
In all these works, the authors study fluid and diffusion approximations for various types of load balancing mechanisms. In each case, stable fixed points are identified for the LLN limit and the interchangeability of limit property is established.

\subsection{Notation}
The following notation will be used. 
The space of probability measures on a Polish space $\SSS$, equipped with the topology of weak convergence, will be denoted by $\clp(\SSS)$.  When $\SSS = \NN_0$ we will metrize $\clp(\SSS)$ with the metric $d_0$ defined as
$$d_0(\mu, \nu) \doteq \sum_{j=0}^{\infty} \frac{|\mu(j)-\nu(j)|}{2^j}, \; \mu, \nu \in \clp(\NN_0).$$
For $\SSS$-valued random variables $X$, $X_n$, $n\ge 1$, convergence in distribution of $X_n$ to $X$ as $n\to \infty$ will be denoted as $X_n \Rightarrow X$.
Let $\bfell_2 = \{(a_j)_{j=0}^\iy|\sum_{j=0}^\iy a_j^2<\iy\}$  be the Hilbert space of square summable real sequences.
Similarly, $\bfell_1=\{(a_j)_{j=0}^\iy|\sum_{j=0}^\iy|a_j|<\iy\}$ is the Banach space of real summable sequences.
For a real number $a$, $(a)_+$ will denote the positive part of $a$. For $a,b\in\RR$ we define $a\wedge b = \min(a,b)$ and $a\vee b = \max(a,b)$.

\section{Model and Results}\label{sec:main}
We consider a system with $n$ servers each with its own infinite capacity queue. In all, there are $I(n)$ equally sized files  stored over the $n$ servers. Each file is stored in equally sized pieces at $L$ servers such that any $k$ pieces can reconstruct the original file. The files are distributed such that each combination of $L$ servers has exactly $c$ files. This, in particular, implies $I(n)=c\binom{n}{L}$. Jobs arrive from outside  according to a Poisson process with rate $n\lambda$ and request one of the $I(n)$ files uniformly at random. Such a request corresponds to selection of one of the $\binom{n}{L}$ sets of $L$ servers, chosen uniformly at random, which is the set of servers containing the pieces of the requested file. The job is then routed to the $k$ shortest queues among this set of $L$ servers. Each server processes queued jobs according to the first-in-first-out (FIFO) discipline. Processing times at each server are mutually independent and exponentially distribution with mean $k^{-1}$. 

Let $Q^n(t)=\{Q^n_i(t)\}_{i=1}^n$ where $Q^n_i(t)$ represents the length of the $i$-th queue at time $t$ and let $\pi^n(t)=\{\pi^n_i(t)\}_{i\in\NN_0}$ where $\pi_i^n(t)$ represents the proportion of queues with length exactly $i$ at time $t$. This can explicitly be written as
\begin{equation}
\pi^n_i(t)=\frac{1}{n}\sum_{j=1}^n1_{\{Q^n_j(t)=i\}}.\label{eq:eq102}
\end{equation}
It will be convenient to work with the process $u^n(t) = \{u^n_i\}_{i\in\NN_0}$ where $u^n_i(t)$ represents the proportion of queues with length at least $i$. Namely, $u^n_i(t) = \sum_{j=i}^\iy\pi^n_j(t)$.
We will assume for simplicity that $Q^n(0)=q^n$ is nonrandom and thus $\pi^n(0)$ and $u^n(0)$ are nonrandom as well.
We identify $\clp(\NN_0)$ with the infinite dimensional simplex $\cls=\{s\in\RR_+^\iy|\sum_{i=0}^\iy s_i=1\}$ and let $\cls_n=\frac{1}{n}\NN_0^\iy\cap\cls$. 
The spaces $\cls$ and $\cls_n$ can be identified with subsets of $\bar{\clu} = \{u\in\RR_+^\iy|1=u_0\geq u_1\geq \ldots\geq 0\}$ and $\bar{\clu}_n = \{u\in\bar{\clu}|u_i = r_i/n, r_i\in\ZZ\}$, respectively, each endowed with the product metric,
\begin{equation*}
\rho(x,y) \doteq\sum_{j=1}^\iy\frac{|x_j-y_j|}{2^j}.
\end{equation*}  
The identification map $\iota:\cls\to\bar{\clu}$ is defined as $\iota(p)_j\doteq\sum_{k=j}^\iy p_k,\ j\in\NN_0,\ p\in\cls$.
Note that for $p^n,p\in\cls$, $d_0(p^n,p)\to 0$ if and only if $\rho(\iota(p^n),\iota(p))\to0$.
Additionally, note that $\pi^n(t)\in\cls_n$ and $u^n(t)\in\bar{\clu}_n$ for all $t\in[0,T]$. Let $\Sigma=\{\ell=(\ell_i)_{i=1}^L\in\NN^L_0|\ell_1\leq\ell_2\leq\cdots\leq\ell_L\}$ and for $\ell\in\Sigma$ define $\rho_i(\ell)\doteq\sum_{j=1}^L1_{\{\ell_j=i\}},\ i\in\NN_0$. Roughly speaking, $\Sigma$ will represent the set of possible states for $L$ selected queues arranged by non-decreasing queue length. Note that each file will be stored at $L$ servers and that at any given time $t$ the queue lengths of these $L$ servers (up to a reordering) will correspond to an element in $\Sigma$. We will  refer to the elements of $\Sigma$ as ``queue length configurations''. 
Given a configuration $\bfell\in\Sigma$, $\rho_i(\ell)$ gives the number of queues of length $i$ (among the $L$ selected).
From the above description of the system it follows that
the empirical measure process, $\pi^n(t)$, is a continuous time Markov chain (see Section 2 of \cite{budhiraja2017diffusion}) with state space $\cls_n$ and generator
\begin{equation}\label{eqn:generator}
\begin{aligned}
\cll^n g(r)
&= \frac{n\lambda}{\binom{n}{L}}\sum_{\ell\in\Sigma}\left(\prod_{i=0}^\iy\binom{nr_i}{\rho_i(\ell)}\right)\left[g\left(r+\frac{1}{n}\Del_\ell\right)-g(r)\right]\\
&\qquad+ k\sum_{i=1}^\iy nr_i\left[g\left(r+\frac{1}{n}(e_{i-1}-e_i)\right)-g(r)\right],
\end{aligned}
\end{equation}
for $g:\cls_n\to\RR$ where
\begin{equation}\label{eqn:deldef}
\Del_\ell \doteq \sum_{i=1}^k e_{\ell_i+1}-\sum_{i=1}^k e_{\ell_i}
\end{equation}
and for $y\in\NN_0,\ e_y\in\bfell_2$ is a vector with 1 at the $y$-th coordinate and $0$ elsewhere. Here we use the standard conventions that $0^0=\binom{0}{0}=0!=1$, and $\binom{a}{b}=0$ when $a<b$. 
The above generator can be understood as follows. A typical term in the second expression corresponds to a jump as a result of a server, with exactly $i$ jobs queued, completing a job.
The term in the square brackets gives the change in value of $f$ as a result of such a jump and the prefactor $knr_i$ corresponds to the fact that servers process jobs at rate $k$ and there are in all $nr_i$ queues (prior to the jump) with exactly $i$ jobs.
The first expression in \eqref{eqn:generator} corresponds to a jump resulting from an arrival of a job to the system.
Typically, such an arrival makes a request for $L$ servers with queue length configuration
$\ell_1\le \ell_2 \le \cdots\le \ell_L$ and results in the jump $\frac{1}{n}\Del_\ell$. 
The sum in \eqref{eqn:deldef} only goes up to $k$ (instead of $L$) since only the smallest $k$ queues are affected by such a jump.
Since prior to the jump, there are $n r_i$ queues with exactly $i$ jobs, the overall rate associated with the configuration $\ell = \{\ell_1\le \ell_2 \le \cdots\le \ell_L\} \in \Sigma$ equals
$$\frac{n\lambda}{\binom{n}{L}}\left(\prod_{i=0}^\iy\binom{nr_i}{\rho_i(\ell)}\right).$$ 
 
Define, for $r\in\bfell_1$,
\begin{equation}\label{eqn:codingF}
F(r)
\doteq  \lambda L!\sum_{j=0}^\iy \bar{\zeta}^\del(j,r)e_j + k\sum_{j=0}^\iy[r_{j+1}-r_j]e_j+r_0e_0
\end{equation}
where
\begin{equation*}
 \bar{\zeta}^\del(j,r) \doteq \bar{\zeta}(j-1,r)-\bar{\zeta}(j,r)
\end{equation*}
and, adopting the convention that $\sum_{i=b}^ax_i=0$ for $a<b$,
\begin{equation}\label{eqn:zetabardef}
\bar{\zeta}(j,r)
\doteq \sum_{i_1=0}^{k-1}\frac{\left(\sum_{m=0}^{j-1}r_m\right)^{i_1}}{i_1!}\sum_{i_2=1}^{L-i_1}[i_2\wedge(k-i_1)]\frac{(r_j)^{i_2}}{i_2!}\frac{\left(\sum_{m=j+1}^{\iy}r_{m}\right)^{L-i_1-i_2}}{(L-i_1-i_2)!}.
\end{equation}
It was shown in Theorem 2.2 of \cite{budhiraja2017diffusion} that $\pi^n\to\pi$, in probability, in $\DD([0,T]:\cls)$ where $\pi$ is the unique solution to the following ODE,
\begin{equation}\label{eqn:ODE}
\dot{\pi}(t) = F(\pi(t)),\qquad \pi(0) = \pi_0.
\end{equation} 

Following \cite{li2016mean}, let $f\equiv f^{(L,k)}:[0,1]\to\RR$ be defined as
\begin{equation*}
f(x)
\doteq \sum_{i=1}^k\binom{L}{L-k+i}\binom{L-k+i-2}{i-1}(-1)^{i-1}x^{L-k+i}.
\end{equation*}
The following lemma gives a representation for $\bar{\zeta}$ in terms of $f$.
\begin{lemma}\label{lem:equiv}
Fix $r\in\cls$ and let $u=\iota(r)$, i.e. $u_m = \sum_{i=m}^\iy r_i,\ m\in\NN_0$. Then, for $j\in\NN_0$
\begin{equation}\label{eqn:zetatof}
L!\bar{\zeta}(j,r) = f(u_j)-f(u_{j+1}).
\end{equation}
\end{lemma}
Proof of this lemma will be give in Section \ref{sec:equivalenceproof}. It follows from Lemma \ref{lem:equiv} and Theorem 2.2 of \cite{budhiraja2017diffusion} that the law of large number limit of $u^n$ solves the following ODE,
\begin{align}\label{eqn:sODE}
\dot{u}_j(t) = \lambda[ f(u_{j-1}(t))-f(u_{j}(t))]-k[u_j(t)-u_{j+1}(t)],\qquad u(0) = g\in\bar{\clu}.
\end{align}
Consider the queue length distribution $\bar{u}=(\bar{u}_m)_{m\in\NN_0}$ defined recursively through,
\begin{align}\label{eqn:fixedPoint}
\left\{\begin{array}{ll}
\bar{u}_{m+1}= \lambda \frac{f(\bar{u}_m)}{k}&\text{ for }m\in\NN_0\\
\bar{u}_0 = 1 &
\end{array}\right.
\end{align}
We will see in Theorem \ref{thm:stableFluid} that $\bar{u}$ is the unique fixed point of \eqref{eqn:sODE}.
The following result shows that the vector $(\bar{u}_m)_{m\in\NN_0}$ which, roughly speaking, represents the steady state distribution of the queue lengths for large $n$, decays super-exponentially in $m$ with rate determined by $L$ and $k$.
\begin{theorem}\label{thm:decay}
Suppose $\bar{u}$ satisfies \eqref{eqn:fixedPoint}. Then the following upper and lower bounds hold:
\begin{enumerate}
\item[i)] $\bar{u}_m\leq \lambda^{\frac{(L/k)^m-1}{L/k-1}}$ for all $m\in\NN_0$.
\item[ii)]  $\bar{u}_m\geq \lambda^{\frac{(L-k+1)^m-1}{L-k}}$ for all $m\in\NN_0$.
%\item[iii)] There exists a $M(L,k)\in\NN_0$ and $a(L,k)<1$ such that $\bar{u}_m\leq a(L,k)^{\frac{(L-k+1)^m-1}{L-k}}$ for all $m>M(L,k)$.
\end{enumerate}
\end{theorem}
We note that the bounds are tight in the sense that when $k=1$ the upper and lower bounds agree.
Proof of this theorem is given in Section \ref{sec:decayProof}.
Since $f$ is a polynomial it is easy to see that $f(x)=\clo(x^{L-k+1})$ as $x\to0$. Intuitively, it makes sense that the queue length distribution should have an upper bound of the form $\lambda^{\frac{(L-k+1)^m-1}{L-k}}$.
Indeed, we can establish an upper bound of this form for large $m$, however due to the higher order terms in $f$ the bound will not hold for small $m$. In fact,  the threshold for a large enough $m$ will depend on $L$ and $k$. Furthermore, the coefficient of $x^{L-k+1}$ in $f$ depends on $L$ and $k$ and, using its form, it can be shown that the upper bound (for large $m$) will be of the form $a^{\frac{(L-k+1)^m-1}{L-k}}$ where $a$ depends on $L$ and $k$.
Recall that the routing scheme considered here corresponds to the well-known ``Power-of-$d$'' or super market model when $L=d$ and $k=1$. The above result reduces to results in \cite{graham2000chaoticity} and \cite{vvedenskaya1996queueing} in this case.

Following \cite{vvedenskaya1996queueing}, define
\begin{equation*}
v_j(u)=\sum_{i=j}^\iy u_i,\quad u\in\bar{\clu}.
\end{equation*}
Let
$\clu \doteq \{u\in\bar{\clu}|v_1(u)<\iy\}$ and note that this can be identified with the space of probability measures on $\NN_0$ with finite first moment.
The space $\clu$ is endowed with the topology inherited from $\bar{\clu}$.
We now characterize the long time behavior of the law of large number limit. 
Note that $\bar{u}\in\clu$.
The next theorem shows that $\bar{u}$ is the unique fixed point in $\clu$ for the system defined by \eqref{eqn:sODE} and this fixed point is, in fact, stable.
\begin{theorem}\label{thm:stableFluid}
Suppose $\lambda <1$ and $u$ is a solution to \eqref{eqn:sODE} with $g\in\clu$. Then
\begin{enumerate}
\item[i)] $u(t)\in\clu$ for all $t$.
\item[ii)] For each $j\in\NN_0$, $\lim_{t\to\iy}(u_j(t)-\bar{u}_j)=0$ and thus $\lim_{t\to\iy}\rho(u(t),\bar{u})=0$. In particular, $\bar{u}$ is the unique fixed point of \eqref{eqn:sODE} in $\clu$.
\end{enumerate}
\end{theorem}
The proof of this theorem will be given in Section \ref{sec:stability}. 

From Proposition 1 of \cite{li2016mean} the process $Q^n$ is positive recurrent and, thus, has a unique invariant distribution $\ti\cll_n\in\clp(\NN_0^n)$. Note that $\ti\cll_n$ can be identified with a measure $\cll_n\in\clp(\bar{\clu}_n)$ which is an invariant measure for $u^n$. Furthermore, for any $t\geq0$, $u^n(t)$ can be mapped to $\ti Q^n(t)\in\NN_0$ which is equal (up to a relabeling) to $Q^n(t)$. Due to symmetry, $Q^n$ and $\ti Q^n$ must have the same invariant distribution. Therefore $\cll_n$ is the unique invariant measure for $u^n$. The following result shows that this invariant measure converges, as $n\to\iy$, to the Dirac measure concentrated at $\bar{u}$.
\begin{theorem}\label{thm:interchange}
Let $\cll_n$ be the unique invariant distribution for the process $u^n$. Then $\cll_n\Rightarrow \del_{\bar{u}}$. Furthermore, we have
\begin{align*}
\lim_{n\to\iy}\lim_{t\to\iy}\E u^n(t) = \bar{u}.
\end{align*}
\end{theorem}
Proof of Theorem \ref{thm:interchange} is given in Section \ref{sec:interchange}.

\section{Proof of Lemma \ref{lem:equiv}}\label{sec:equivalenceproof}
The result will follow upon verifying,
\begin{align}\label{eqn:equiv1}
&L!\bar{\zeta}(m,r) = \sum_{\ell=1}^k\sum_{i_1=0}^{\ell-1}\binom{L}{i_1}\left(1-u_m\right)^{i_1}\sum_{i_2=\ell-i_1}^{L-i_1}\binom{L-i_1}{i_2}r_m^{i_2}u_{m+1}^{L-i_1-i_2}
\end{align}
and, for $\ell = \{1,\ldots,k\}$,
\begin{align}\label{eqn:equiv2}
&\sum_{j=m}^\iy\sum_{i_1=0}^{\ell-1}\binom{L}{i_1}\left(1-u_j\right)^{i_1}\sum_{i_2=\ell-i_1}^{L-i_1}\binom{L-i_1}{i_2}(r_j)^{i_2}u_{j+1}^{L-i_1-i_2}
= \sum_{j=L-\ell+1}^L\binom{L}{j}u^j_m(1-u_m)^{L-j}.
\end{align}
These equations can be interpreted as follows. Suppose the occupancy measure is in state $r$. Roughly speaking, a typical term in the outside summation on the RHS of \eqref{eqn:equiv1}, denoted as $p_m(\ell)$, corresponds to the probability that the $\ell$-th largest out of $L$ randomly selected queues is of length $m$. Then \eqref{eqn:equiv1} states that the rate of jobs being routed into queues of length $m$ is equal to the sum $\sum_{\ell=1}^kp_m(\ell)$. 
Recall that a file request will correspond to a queue length configuration $a_1\leq a_2\leq\cdots\leq a_L$, where $a_j$ corresponds to the length of the $j$-th largest queue. 
A typical term in the outside summation on the LHS of \eqref{eqn:equiv2}, denoted $\ti p_\ell(j)$, corresponds to the probability that $a_\ell = j$ . 
Terms in the summation on the RHS of \eqref{eqn:equiv2}, denoted $q_m(j)$, correspond to the probability that $a_{j-1}<m\leq a_{j}$. 
The expression \eqref{eqn:equiv2} then states that $\sum_{j=m}^\iy\ti p_\ell(j) = \sum_{j=L-\ell+1}^L  q_m(j)$.
Once these equalities are established the remainder of the argument follows as in Appendix B of \cite{li2016mean} which argues that $\sum_{\ell=1}^k\sum_{j=L-\ell+1}^L q_m(j)=f(u_m)$. Combining this fact with \eqref{eqn:equiv1} and \eqref{eqn:equiv2} then gives,
\begin{align*}
L!\bar{\zeta}(m,r) &= \sum_{\ell=1}^kp_m(\ell)
= \sum_{\ell=1}^k\left[\sum_{j=m}^\iy\ti p_\ell(j)-\sum_{j=m+1}^\iy\ti p_\ell(j)\right]
= \sum_{\ell=1}^k\left[\sum_{j=L-\ell+1}^L q_m(j)-\sum_{j=L-\ell+1}^L q_{m+1}(j)\right]\\
&= f(u_m)-f(u_{m+1})
\end{align*}
which proves the result.

We now prove the two equalities. First consider \eqref{eqn:equiv1}. By rearranging and collecting combinatorial terms we can write
\begin{align}\label{eqn:comb1}
L!\bar{\zeta}(m,r)=\sum_{i_1=0}^{k-1}\sum_{i_2=1}^{L-i_1}[i_2\wedge(k-i_1)]\binom{L}{i_1}\binom{L-i_1}{i_2}\left(1-u_m\right)^{i_1}r_m^{i_2}u_{m+1}^{L-i_1-i_2}.
\end{align}
Note that the RHS in \eqref{eqn:comb1} can be written as
\begin{align}\label{eqn:comb2}
\begin{split}
&\sum_{i_1=0}^{k-1}\sum_{i_2=1}^{L-i_1}\sum_{\ell=1}^{[i_2\wedge(k-i_1)]}\binom{L}{i_1}\binom{L-i_1}{i_2}\left(1-u_m\right)^{i_1}r_m^{i_2}u_{m+1}^{L-i_1-i_2}\\
&\qquad=
\sum_{i_1=0}^{k-1}\binom{L}{i_1}\left(1-u_m\right)^{i_1}\sum_{i_2=1}^{L-i_1}\sum_{\ell=i_1+1}^{(i_1+i_2)\wedge k}\binom{L-i_1}{i_2}r_m^{i_2}u_{m+1}^{L-i_1-i_2}.
\end{split}
\end{align}
We then exchange the order of summations as follows
\begin{align}\label{eqn:comb3}
\begin{split}
&\sum_{i_1=0}^{k-1}\binom{L}{i_1}\left(1-u_m\right)^{i_1}\sum_{i_2=1}^{L-i_1}\sum_{\ell=i_1+1}^{(i_1+i_2)\wedge k}\binom{L-i_1}{i_2}r_m^{i_2}u_{m+1}^{L-i_1-i_2}\\
&\qquad=\sum_{i_1=0}^{k-1}\sum_{\ell=i_1+1}^k\binom{L}{i_1}\left(1-u_m\right)^{i_1}\sum_{i_2=\ell-i_1}^{L-i_1}\binom{L-i_1}{i_2}r_m^{i_2}u_{m+1}^{L-i_1-i_2}\\
&\qquad=\sum_{\ell=1}^k\sum_{i_1=0}^{\ell-1}\binom{L}{i_1}\left(1-u_m\right)^{i_1}\sum_{i_2=\ell-i_1}^{L-i_1}\binom{L-i_1}{i_2}r_m^{i_2}u_{m+1}^{L-i_1-i_2}.
\end{split}
\end{align}
Combining \eqref{eqn:comb1}, \eqref{eqn:comb2}, and \eqref{eqn:comb3} gives \eqref{eqn:equiv1}.

We now prove \eqref{eqn:equiv2}. Fix $\ell\in\{1,\ldots,k\}$ and note that
\begin{align}\label{eqn:split1}
\sum_{j=L-\ell+1}^L\binom{L}{j}u^j_m(1-u_m)^{L-j}
= \sum_{i_1=0}^{\ell-1}\binom{L}{i_1}(1-u_m)^{i_1}u^{L-i_1}_m.
\end{align}
Then, applying the binomial theorem to $u^{L-i_1}_m=\left(r_m+u_{m+1}\right)^{L-i_1}$, \eqref{eqn:split1} becomes
\begin{align}\label{eqn:split2}
\sum_{i_1=0}^{\ell-1}\binom{L}{i_1}(1-u_m)^{i_1}u^{L-i_1}_m
= \sum_{i_1=0}^{\ell-1}\binom{L}{i_1}(1-u_m)^{i_1}\sum_{i_2=0}^{L-i_1}\binom{L-i_1}{i_2}r_m^{i_2}u^{L-i_1-i_2}_{m+1}
\end{align}
which, by breaking up the summation indexed by $i_2$, can be rewritten as
\begin{align}\label{eqn:split3}
\begin{split}
&\sum_{i_1=0}^{\ell-1}\binom{L}{i_1}(1-u_m)^{i_1}\sum_{i_2=\ell-i_1}^{L-i_1}\binom{L-i_1}{i_2}r_m^{i_2}u^{L-i_1-i_2}_{m+1}\\
&\qquad+ \sum_{i_1=0}^{\ell-1}\binom{L}{i_1}(1-u_m)^{i_1}\sum_{i_2=0}^{\ell-i_1-1}\binom{L-i_1}{i_2}r_m^{i_2}u^{L-i_1-i_2}_{m+1}.
\end{split}
\end{align}
Now consider the second term in \eqref{eqn:split3}. By relabeling the indices we get
\begin{align}\label{eqn:split4}
\begin{split}
&\sum_{i_1=0}^{\ell-1}\binom{L}{i_1}(1-u_m)^{i_1}\sum_{i_2=0}^{\ell-i_1-1}\binom{L-i_1}{i_2}r_m^{i_2}u^{L-i_1-i_2}_{m+1}\\
&\qquad= \sum_{i_1=L-\ell+1}^{L}\binom{L}{i_1}u_{m+1}^{i_1}\sum_{i_2=0}^{L-i_1}\binom{L-i_1}{i_2}(1-u_m)^{i_2}r_m^{L-i_1-i_2}\\
&\qquad= \sum_{i_1=L-\ell+1}^{L}\binom{L}{i_1}u_{m+1}^{i_1}(1-u_{m+1})^{L-i_1}
\end{split}
\end{align}
where the second equality follows from the binomial theorem. It then follows from \eqref{eqn:split1}-\eqref{eqn:split4} that for any $m'>m$
\begin{align*}
&\sum_{j=L-\ell+1}^L\binom{L}{j}u^j_m(1-u_m)^{L-j}\\
&\qquad=\sum_{j=m}^{m'}\sum_{i_1=0}^{\ell-1}\binom{L}{i_1}(1-u_j)^{i_1}\sum_{i_2=\ell-i_1}^{L-i_1}\binom{L-i_1}{i_2}r_j^{i_2}u^{L-i_1-i_2}_{j+1}+ \sum_{j=L-\ell+1}^{L}\binom{L}{j}u_{m'}^{j}(1-u_{m'})^{L-j}.
\end{align*}
The result then follows upon sending $m'\to\iy$.
\qed

\section{Proof of Theorem \ref{thm:decay}}\label{sec:decayProof}
It is proved in Lemma 5 of \cite{li2016mean} that $f(x)\leq kx^{L/k}$ and thus $(i)$ is immediate from \eqref{eqn:fixedPoint}.

We now verify $(ii)$. From \eqref{eqn:fixedPoint} it suffices to show that $f(x)\geq kx^{L-k+1}$ for $x\in[0,1]$. Since both sides of the inequality evaluate to zero at $x=0$, it is equivalent to show 
\begin{align}\label{eqn:hcond}
h(x) \doteq \frac{1}{k}\frac{f(x)}{x^{L-k+1}}\geq1\text{ for }x\in(0,1].
\end{align}
Note that $f(1)=k$ (cf. Lemma 2 of \cite{li2016mean}), and thus $h(1) = 1$. It follows that $h'(x)\leq 0,\ x\in(0,1]$ is sufficient for verifying \eqref{eqn:hcond}. 
Taking the derivative of $h$ gives
\begin{align}\label{eqn:hprime}
h'(x) = \frac{1}{k}\frac{x f'(x)-(L-k+1)f(x)}{x^{L-k+2}}.
\end{align}
Denoting $x f'(x)-(L-k+1)f(x)$ as $w(x)$, we note that in order to show $h'(x)\leq 0$ one must only verify $w(x)\leq 0$. One can verify (cf. (68) and (69) of \cite{li2016mean}) that $w$ can be expressed as follows 
\begin{align*}
w(x) \doteq \sum_{\ell=0}^{k-1}\binom{k-1}{\ell}(-1)^\ell\frac{1}{L-k+\ell}\frac{\ell}{L-k+\ell+1}x^{L-k+\ell+1}.
\end{align*}
Therefore
\begin{align*}
w''(x) = x^{L-k-1}\sum_{\ell=0}^{k-1}\binom{k-1}{\ell}\ell (-x)^{\ell}
= -(k-1)x^{L-k}(1-x)^{k-2}
\end{align*}
and so $w''(x)\leq 0$ for all $x\in[0,1]$. Noting that $w'(0)=0$, it follows that $w'(x)\leq 0$ for all $x\in(0,1]$ and since $w(0)=0$, $w(x)\leq 0$ for $x\in(0,1]$. This verifies \eqref{eqn:hcond}.
\qed

\section{Proof of Theorem \ref{thm:stableFluid}}\label{sec:stability}
In this section we present the proof of Theorem \ref{thm:stableFluid}.
Namely, for every $g\in\clu$, the solution $u$ of \eqref{eqn:sODE} satisfies $u(t)\in\clu$ for all $t\geq0$ (Lemma \ref{lem:momentBound}), and there is a unique fixed point to \eqref{eqn:sODE} in $\clu$ defined by \eqref{eqn:fixedPoint} which is asymptotically stable.
The argument follows along the lines of the proof of Theorem 1 of \cite{vvedenskaya1996queueing} (cf. Lemmas 1-7 therein).
The key difference is that the the term $\lambda[f(u_{i-1})-f(u_{i})]$ appears in the differential equation instead of $\lambda[u_{i-1}^2-u_i^2]$.
As we will see, this difference can be handled using the properties of $f$ shown in Lemma 2 of \cite{li2016mean}. 
Specifically, we will use the facts that $f(0)=0$, $f(1) = k$, $f$ is strictly increasing, convex, and differentiable with derivative bounded by $L$.

We first consider a truncated version of \eqref{eqn:sODE}. Fix $K\in\NN$, $c\geq 0$, and consider the following boundary value problem
\begin{equation}\label{eqn:truncSys}
\left\{\begin{array}{ll}
\dot{s}_j(t)
&= \lambda [f(s_{j-1}(t))-f(s_{j}(t))]-k[s_j(t)-s_{j+1}(t)],\quad j = 1,\ldots,K\\
s_0(t) &= 1\\
s_j(0) &= g_j,\quad j= 1,\ldots, K
\end{array}\right.
\end{equation}
with
\begin{equation}\label{eqn:truncInit}
s_{K+1}(t) = c.
\end{equation}
The following two lemmas giving monotonicity and uniqueness properties for the truncated system will be used to extend the same properties to the full system in Lemma \ref{lem:monotone}.
\begin{lemma}\label{lem:truncInc}
Suppose $s$ is a solution to \eqref{eqn:truncSys}-\eqref{eqn:truncInit} with initial conditions satisfying
\begin{align}\label{eqn:decreasingInit}
1=g_0\geq g_1\geq\ldots\geq g_K\geq g_{K+1}=c.
\end{align}
Then
\begin{align}\label{eqn:decreasingEqn}
1=s_0(t)\geq s_1(t)\geq\ldots\geq s_K(t)\geq s_{K+1}(t)=c
\end{align}
for all $t\geq0$.
\end{lemma}
\begin{proof}
Since solutions to \eqref{eqn:truncSys}-\eqref{eqn:truncInit} depend continuously on the initial conditions we can take the inequalities in \eqref{eqn:decreasingInit} to be strict, without loss of generality. Let $t_0$ be the first time that an equality appears in \eqref{eqn:decreasingEqn}. Since $s_0>s_{K+1}$, then there exists an $i\in\{1,\ldots,K\}$ such that either $s_{i-1}(t_0)>s_i(t_0)=s_{i+1}(t_0)$ or $s_{i-1}(t_0)=s_i(t_0)>s_{i+1}(t_0)$. 
In the former case, since $f$ is strictly increasing, $\dot{s}_i(t_0) = \lambda[f(s_{i-1}(t_0))-f(s_{i}(t_0))]>0$ and $\dot{s}_{i+1}(t_0) = k[s_{i+2}(t_0)-s_{i+1}(t_0)]\leq 0$ if $i<K$ and $s_{i+1}(t_0)=0$ if $i=K$, both of which contradict the assumption that $s_i(t)>s_{i+1}(t)$ for $t<t_0$. The latter case follows from a similar argument.
\end{proof}

\begin{lemma}\label{lem:truncMon}
Let $\{s_i^{(1)}\}_{i=0}^K$ and $\{s_i^{(2)}\}_{i=0}^K$ solve \eqref{eqn:truncSys} and be such that $s^{(1)}_i(0)\geq s^{(2)}_i(0)$ for all $i=1,2,\ldots,K$. If, in addition, $s^{(1)}_{K+1}(t)\geq s^{(2)}_{K+1}(t)$ for all $t\geq0$ then $s^{(1)}_i(t)\geq s^{(2)}_k(t)$ for all $i=1,2, \ldots,K,K+1$ and all $t\geq0$.
\end{lemma}
\begin{proof}
Again, assume without loss of generality that the inequalities are strict. I.e. $s^{(1)}_i(0)> s^{(2)}_i(0)$ for all $i=1,2,\ldots,K$ and $s^{(1)}_{K+1}(t)> s^{(2)}_{K+1}(t)$, for all $t\geq0$. Suppose the first time equality appears is at time $t_0$. If $j\in\{1,\ldots,K\}$ is the largest index such that $s^{(1)}_j(t_0)= s^{(2)}_j(t_0)$ then, since $f$ is strictly increasing, 
\begin{align*}
\dot{s}^{(1)}_j(t_0)-\dot{s}^{(2)}_j(t_0)
& = \lambda [f(s_{j-1}^{(1)}(t_0))-f(s^{(2)}_{j-1}(t_0))]+k[s^{(1)}_{j+1}(t_0)-s^{(2)}_{j+1}(t_0)]\\
&\geq k[s^{(1)}_{j+1}(t_0)-s^{(2)}_{j+1}(t_0)]
> 0
\end{align*}
which contradicts the assumption $s^{(1)}_j(t)>s^{(2)}_j(t)$ for $t<t_0$.
\end{proof}

Note that Lemma \ref{lem:truncMon}, in particular, shows that there is a unique solution to \eqref{eqn:truncSys}-\eqref{eqn:truncInit}.
We now consider the full system \eqref{eqn:sODE}. In the following lemma we show that the full system can be constructed as the limit of the sequence of truncated systems defined through \eqref{eqn:truncSys}.
\begin{lemma}\label{lem:trunctoLim}
Let $g\in\bar{\clu}$. 
\begin{enumerate}
\item[i)] There exists a unique solution to \eqref{eqn:sODE} in $\bar{\clu}$. 
\item[ii)] This solution can be obtained as the limit as $K\to\iy$ of solutions to the truncated systems \eqref{eqn:truncSys}-\eqref{eqn:truncInit} associated with $c=0$.
\end{enumerate}
\end{lemma}
\begin{proof}
Part $(i)$ follows immediately from Lemma \ref{lem:equiv} and Proposition 2.1 of \cite{budhiraja2017diffusion}. Let $s^K(t),\ K=1,2,\ldots$ denote solutions to \eqref{eqn:truncSys} with $s^K_{K+1}(t)=0$. It follows from Lemma \ref{lem:truncInc} that $s^{K+1}_{K+1}(t)\geq s^{K}_{K+1}(t)=0$ and from Lemma \ref{lem:truncMon} that for fixed $t$ and $i\leq K$, $s^{K+1}_i(t)\geq s_i^K(t)$. It follows that $\lim_{K\to\iy}s^K_i(t)=s_i(t)$ exists, $s(t)\in\bar{\clu}$, and $s_i$ satisfies \eqref{eqn:sODE} which proves $(ii)$.
\end{proof}

\begin{lemma}
Let $u$ be a solution to \eqref{eqn:sODE} taking values in $\bar{\clu}$. Then the following estimate holds for all $t$,
\begin{equation}\label{eqn:taylorEst}
u_j(t)\leq \sum_{i=0}^j\frac{u_i(0)(\lambda k t)^{j-i}}{(j-i)!},\qquad j\in\NN_0.
\end{equation}
\end{lemma}
\begin{proof}
The lemma follows from using an inductive argument. Note that the inequality is immediate for $j=0$. 
Suppose now that \eqref{eqn:taylorEst} holds for $j-1$, for some $j\geq 1$.
Then, since $f(0)=0$, $f(1) = k$, and $f$ is convex on $[0,1]$, it follows from \eqref{eqn:sODE} that
\begin{equation*}
\dot{u}_j(t)\leq \lambda f(u_{j-1}(t))\leq \lambda k u_{j-1}(t).
\end{equation*}
Since \eqref{eqn:taylorEst} holds for $j-1$ by our inductive hypothesis we have, by integrating over $t$ on both sides of the above inequality, that \eqref{eqn:taylorEst} also holds for $j$. The result follows. 
\end{proof}

\begin{lemma}\label{lem:momentBound}
Let $u$ be a solution to \eqref{eqn:sODE} taking values in $\bar{\clu}$. If $u(0)\in\clu$, then $u(t)\in\clu$ for all $t\geq 0$. Furthermore, $v_1(u(t))\leq \exp(\lambda k t)[1+v_1(u(0))]$.
\end{lemma}
\begin{proof}
This follows immediately from the estimate \eqref{eqn:taylorEst}.
\end{proof}

The following monotonicity property of the full system \eqref{eqn:sODE} is an immediate consequence of Lemma \ref{lem:truncMon} and part $(ii)$ of Lemma \ref{lem:trunctoLim}.
\begin{lemma}\label{lem:monotone}
Let $u^{(1)}$ and $u^{(2)}$ be a solutions to \eqref{eqn:sODE} in $\bar{\clu}$ with $u^{(1)}_j(0)\geq u^{(2)}_j(0)$ for all $j\in\NN_0$. Then $u^{(1)}_j(t)\geq u^{(2)}_j(t)$ for all $j=\NN_0$ and all $t\geq0$.
\end{lemma}
%\begin{proof}
%This follows from Lemmas \ref{lem:truncMon} and \ref{lem:trunctoLim}.
%\end{proof}

With the above lemmas we can now complete the proof of Theorem \ref{thm:stableFluid}.
\begin{proof}[Proof of Theorem \ref{thm:stableFluid}]
Part $(i)$ of the theorem was shown in Lemma \ref{lem:momentBound}. 

Now consider part $(ii)$.
Suppose $g_i\leq \bar{u}_i,\ i\in\NN_0$. 
Then from Lemma \ref{lem:monotone}, it follows that $v_1(u(t))\leq\sum_{i=1}^\iy\bar{u}_i<\iy$. 
If instead, $g_i\geq \bar{u}_i,\ i\in\NN_0$ then from \eqref{eqn:fixedPoint} and noting that $\bar{u}_0=1$ and $f(1)=k$, we have that $\bar{u}_1 = \lambda f^{(L_,k)}(1)/k = \lambda$. 
Thus, from Lemma \ref{lem:monotone} once more, $u_1(t)\geq \lambda$ for all $t\geq 0$, from which it follows that
\begin{equation*}
\dot{v}_1(u(t)) = \lambda f(1)-ku_1(t)
= k(\lambda- u_1(t))
\leq 0.
\end{equation*}
Therefore, in both cases $v_1(u(t))$ is uniformly bounded in $t$. Assume for now that we are in one of these two cases.

We now prove that
\begin{align}\label{eqn:momIntBound}
\int_0^\iy|u_k(t)-\bar{u}_k|dt <\iy
\end{align}
for each $k$. Noting that $f$ has derivative bounded by $L$ (cf. Lemma 2 of \cite{li2016mean}) it will then follow that, for each of these two cases we have the desired convergence
\begin{align}\label{eqn:convergenceK}
\lim_{t\to\iy}|u_k(t)-\bar{u}_k|=0, \text{ for all }k\in\NN_0.
\end{align}
From this, convergence for an arbitrary initial condition will follow on noting that from Lemma \ref{lem:monotone}, $u^{-}(t)\leq u(t)\leq u^{+}(t)$ where $u^{-}$ and $u^{+}$ are the solutions to \eqref{eqn:sODE} with $u^{-}_k(0)= g_k\wedge \bar{u}_k$ and $u^{+}_k(0)= g_k\vee \bar{u}_k$.
Finally, we prove \eqref{eqn:momIntBound} using an inductive argument. It is clear that \eqref{eqn:momIntBound} holds for $k=0$. Now suppose \eqref{eqn:momIntBound} is true for $k-1$, for some $k\geq 1$. Then
\begin{align*}
\dot{v}_k(u(t)) = \lambda f(u_{k-1}(t))-ku_k(t) = \lambda [f(u_{k-1}(t))-f(\bar{u}_{k-1})]-k[u_k(t)-\bar{u}_k]
\end{align*}
and thus
\begin{align*}
v_k(u(t))-v_k(g) = \int_0^t\left( \lambda [f(u_{k-1}(s))-f(\bar{u}_{k-1})]-k[u_k(s)-\bar{u}_k]\right)ds.
\end{align*}
Note that since $v_k(u(t))\leq v_1(u(t))$ we must have that $v_k(u(t))-v_k(g)$ is uniformly bounded in $t$.
From the inductive assumption and appealing again to the boundedness of the first derivative of $f$ it follows that $\sup_{t\in[0,\iy)}\int_0^t \lambda [f(u_{k-1}(s))-f(\bar{u}_{k-1})]ds<\iy$. Therefore \eqref{eqn:momIntBound} is satisfied for $k$ which completes the proof.
\end{proof}

\section{Proof of Theorem \ref{thm:interchange}}\label{sec:interchange}
Note that $\cll_n$ is a probability measure on the set  $\bar{\clu}$ which is a compact set in the product topology. 
Thus, $\{\cll_n\}_{n\in\NN}$ is a tight sequence in $\clp(\bar{\clu})$. Let $\{\cll_{n_k}\}_{k\in\NN}$ be a weakly convergent subsequence with limit point $\cll$. 
Suppose $u^{n_k}(0)$ is distributed according to $\cll_{n_k}$ (we write $u^{n_k}(0)\sim\cll_{n_k}$). 
Then $u^{n_k}(t)\sim\cll_{n_k}$ for all $t\geq 0$. 
By a minor modification of the proof of Theorem 2.2 of \cite{budhiraja2017diffusion} it follows now that $u_{n_k}\Rightarrow u$ in $\DD([0,T],\clu)$ where $u$ solves the ODE \eqref{eqn:sODE} a.s. 
Theorem 2.2 of \cite{budhiraja2017diffusion} proves such a result for the case where the initial occupancy measure $u(0)$ is deterministic. 
However, the extension to the case where the initial conditions are stochastic is straight forward.
Since at any time $t$, $u_{n_k}(t)\sim\cll_{n_k}$, it follows that $u(t)\sim\cll$.
From the fact that $\bar{u}$ is the unique fixed point of \eqref{eqn:sODE} it follows now that $\cll=\del_{\bar{u}}$ and thus $\del_{\bar{u}}$ must be the limit point of every convergent subsequence. This completes the proof of the first statement in Theorem \ref{thm:interchange}. The second statement is immediate on noting that for all $k\in\NN_0$, $\E u^n_k(t)\to\int_\clu u_kd\cll^n(u)$ as $t\to\iy$.
\qed

\bigskip
\noindent {\bf Acknowledgements.}  Research supported in part by the National Science Foundation (DMS-1016441, DMS-1305120) and the Army Research Office (W911NF-14-1-0331). In addition, I would like to thank Professor Amarjit Budhiraja for discussions and advice.
\bibliographystyle{plain}

{\sc
\bigskip
\noindent
E. Friedlander\\
Department of Statistics and Operations Research\\
University of North Carolina\\
Chapel Hill, NC 27599, USA\\
email: ebf2@live.unc.edu

}
\end{document}